\documentclass[12pt]{amsart}
\usepackage{amssymb,latexsym}
\usepackage{enumerate}
\usepackage{enumitem}
\usepackage{hyperref}
\usepackage{fullpage}
\usepackage{fixme}
\fxsetup{
    status=draft,
    author=,
    layout=margin,
    theme=color
}

\makeatletter \@namedef{subjclassname@2010}{%
  \textup{2010} Mathematics Subject Classification}
\makeatother

\newcounter{thm} \numberwithin{thm}{section}
\newtheorem{Theorem}[thm]{Theorem}

\newtheorem{Lemma}[thm]{Lemma}
\newtheorem{Corollary}[thm]{Corollary}

\newtheorem*{Claim}{Claim}

	\newcommand{\RR}[0]{\mathbb R}

\newcommand{\eps}[0]{\varepsilon}

\author[B. Hanson]{Brandon Hanson} \address{University of Georgia\\
Athens, GA, USA}
\email{brandon.w.hanson@gmail.com}
\author[O. Roche-Newton]{Oliver Roche-Newton} \address{Johann Radon Institute for Computational and Applied Mathematics\\
Linz, Austria}
\email{o.rochenewton@gmail.com}
\author[M. Rudnev]{Misha Rudnev} \address{University of Bristol\\ 
Bristol, UK}
\email{misharudnev@gmail.com}
\date{}

\begin{document}

\baselineskip=17pt

\title{Higher convexity and iterated sum sets}

\date{}
\maketitle

\begin{abstract}

Let $f$ be a smooth real function with strictly monotone first  $k$ derivatives.
We show that for a finite set $A$, with $|A+A|\leq K|A|$, 
\[
|2^kf(A)-(2^k-1)f(A)|\gg_k |A|^{k+1-o(1)}/K^{O_k(1)}.
\] 
We deduce several new sum-product type implications, e.g. that $A+A$ being small implies unbounded growth for a many enough times iterated product set $A \cdots A$.
\end{abstract}

\section{Introduction}
When $A$ is a subset of an Abelian group, define its sumsets via
\[A+A=\{x_1+x_2:x_1,x_2\in A\},\] 
and 
\[mA-nA=\{(x_1+\cdots+x_m)-(x_{m+1}+\cdots+x_{m+n}):x_1,\ldots,x_{m+n}\in A\}.\] Additive combinatorics is interested in the cardinality and structure of sumsets. Now suppose $A=\{a_1<\ldots<a_N\}$ is a set of real numbers, listed in increasing order. The sequence of first differences has terms 
\[
a_{i+1}-a_i,\,\,\,\, 1 \leq i \leq N-1,
\] and when this sequence is strictly increasing, we say $A$ is \emph{convex}. Erd\H{o}s posed the problem of estimating $|A+A|$ when $A$ is convex. It is no loss of generality to assume $A=f([N])$, where $[N]=\{1,\ldots,N\}$ and $f$ is a differentiable function with strictly increasing first derivative.

The first non-trivial lower bound was given by Hegyv\'ari \cite{H}. Following Elekes's introduction of incidence geometry to sum-product type problems in \cite{Elekes}, results concerning sum set estimates for convex $A$ were generalised by Elekes, Nathanson and Ruzsa in \cite{ENR}:\footnote{We shall use throughout the notation $X\ll Y$ to mean $X\leq CY$ for some absolute constant $C$. Writing $X\ll_k Y$ means there is a constant $C_k$ which depends only on $k$ such that $X\leq C_kY$. All the sets, denoted by uppercase letters are finite, of cardinality $|\cdot|$.}
\begin{Theorem}[Elekes, Nathanson and Ruzsa]\label{ENR}
Let $A,B,C \subset \mathbb R$ and let $f: \mathbb R \rightarrow \mathbb R$ be a strictly convex function. Then,
\begin{equation} \label{ENRgen}
|A+B|^2|f(A)+C|^2 \gg |A|^3|B||C|.
\end{equation}
In particular, if $A$ is a convex set then $|A+A| \gg |A|^{3/2}$.
\end{Theorem}

The current record is due to Shkredov \cite{Shk},  building up on work of Schoen and Shkredov from \cite{SS}. If $A$ is convex, then
\[|A +A| \gg \frac{|A|^{58/37}}{(\log |A|)^{20/37}},\qquad |A -A| \gg \frac{|A|^{8/5}}
{\log^{2/5}|A|} \,.\]
The right-hand side of both inequalities  has been improved by Olmezov \cite{Ol} to $|A|^{5/3}$ under additional assumptions on the signs of third and fourth derivatives of the convex function $f$, with $A=f([N])$. Olmezov develops on the method of Garaev \cite{G}, which is elementary: it relies only on basic combinatorics and order. The same is true of this paper and the bulk of \cite{RSSS}, from which we draw our inspiration. 

A related result, appearing in \cite{RSSS}, establishes the following theorem in a simple but remarkably effective manner.
\begin{Theorem}[Ruzsa, Shakan, Solymosi, Szemer\'edi]\label{RSSS}
For a convex $A\subset \mathbb R$, one has  \[|2A-A|\gg|A|^2.\]
\end{Theorem}

In this article we develop upon the elementary considerations underlying the above result. In a sense, Theorem \ref{RSSS} cannot be improved: take $A=\{n^2:1\leq n\leq N\}$. However, if one had a set $A$ which was ``more" convex than the first $N$ perfect squares one might hope for more, and indeed this is the first generalisation we  establish. In what follows, it will be easier to think of $A$ as a sequence, in increasing order. We say $A=\{a_1<\ldots<a_N\}$ is $1$-convex if it is convex, and inductively we say $A$ is $k$-convex if its sequence of first differences $\{a_{i+1}-a_i\}$ are $(k-1)$-convex. Thus, $\{n^2:1\leq n\leq N\}$ is $1$-convex, while $\{n^{3}:1\leq n\leq N\}$ is 2-convex, and $\{2^n:1\leq n\leq N\}$ is $k$-convex for any $k \leq N-1$.

The foundation of our paper is the following  generalisation of Theorem \ref{RSSS}.

\begin{Theorem}\label{BasicHigherConvex}
Let $k\geq 2$ be an integer and let $A$ be a $k$-convex sequence. Then
\[|2^{k}A-(2^{k}-1)A|\gg \frac{|A|^{k+1}}{2^{k^2}}.\]
\end{Theorem}

We note here that a small modification of the forthcoming proof of Theorem \ref{BasicHigherConvex} gives a better dependence on the constant $k$, replacing the $2^{k^2}$ in the denominator with $k!$. However, we present this slightly weaker bound in order to align its proof with that of Theorem \ref{MainHigherConvex}, our main result.

The form of the result of Elekes, Nathanson and Ruzsa stated in \eqref{ENRgen} may be viewed as a generalisation of the sum-product phenomenon. It implies, for instance, that at least one of $|A+A|$ and $|f(A)+f(A)|$ must be large. Indeed, Theorem \ref{ENR} implies that if $A$ is a finite set of real numbers satisfying $|A+A|\leq K|A|$, and $f$ is a convex function, then
\[|f(A)+f(A)|\gg \frac{|A|^{3/2}}{K}.\]

Our main theorem extends Theorem \ref{BasicHigherConvex} into a statement in the spirit of this inequality. For an interval $I$, we say $f: I \rightarrow \mathbb R$ is $0$-convex function if it is strictly monotone on $I$, and in general, $f$ is $k$-convex on $I$ if each of the derivatives $f^{(1)},f^{(2)},\ldots,f^{(k+1)}$ exists and is  non-vanishing on $I$ (hence all but the last one are strictly monotone).
\begin{Theorem}\label{MainHigherConvex}
Let $A$ be a finite set of real numbers contained in an interval $I$ and let $f$ be function which is $k$-convex on $I$ for some $k\geq 1$. Suppose that $|A| >10k$. Then if $|A+A-A|\leq K|A|$, we have
\[
|2^kf(A)-(2^{k}-1)f(A)|\geq \frac{|A|^{k+1}}{(CK)^{2^{k+1}-k-2}(\log |A|)^{2^{k+2}-k-4}},\]
for some absolute constant $C>0$.
\end{Theorem}

By applying Theorem \ref{MainHigherConvex} to various choices of $f$, we obtain theorems about iterated growth in the spirit of the renown Erd\H{o}s-Szemer\'edi sum-product problem. Finite sets $A\subset \RR$ defining few distinct products are known to yield arbitrarily large iterated sumsets for sufficiently many iterations: this constitutes the so-called FPMS (few products, many sums) phenomenon. See e.g. \cite{CH}, \cite{Li}, \cite{BuCr}, \cite{Sh} and the references contained therein. Among these, the results by Shkredov \cite{Sh} apply to non-ordered fields as well, for they are based on geometric incidence estimates which hold true over general fields. In contrast, we make a point of avoiding geometric incidence theory here but use order in an essential way, as is the case in \cite{BuCr}, \cite{CH}, and \cite{Li}. An application of our Theorem \ref{MainHigherConvex} gives the iterated sums growth comparable to \cite{CH} and \cite{Li} but is weaker than \cite{BuCr}. However, the above-mentioned papers rely quite heavily on the Distributive Law, and their results are therefore specific to the FPMS setup. 

Theorem \ref{MainHigherConvex}, however, is sufficiently general so as to produce bounds for other sum-product type iterated growth questions that have not, to our knowledge, been established before. \begin{Corollary}\label{Applications} For all $k\geq 1$ and $\eps>0$, there exists $\delta=\delta(k,\eps)>0$ such that, for any sufficiently large finite set of positive real numbers $A$, the following hold:
\begin{enumerate}
    \item If $|A+A|\leq |A|^{1+\delta}$, then we have \[|(A\stackrel{2^{k-1}}{\cdots}A)/(A\stackrel{{2^{k-1}-1}}{\cdots}A)|\gg |A|^{k-\eps}.\]
    \item If $|AA|\leq |A|^{1+\delta}$, then we have
    \[|((A+1)\stackrel{2^{k-1}}{\cdots}(A+1))/((A+1)\stackrel{{2^{k-1}-1}}{\cdots}(A+1))|\gg |A|^{k-\eps}.\]
    \item If $|A+A|\leq |A|^{1+\delta}$ and $A^k=\{a^k:a\in A\}$ for some integer $k\geq 2$, then we have \[|2^{k-1} A^k-(2^{k-1}-1)A^k|\gg |A|^{k-\eps}.\]
    The same is true if $a^k$ is replaced by any polynomial in $a$ of degree $k$, or any real (positive or negative) power of $a$ distinct from $0,\ldots,k-1$, provided it is well-defined.
\end{enumerate}
\end{Corollary}
\noindent Further applications in light of some questions existing in the literature are left to Section \ref{sec:applications}.

We conclude this introduction with a few remarks. First, Theorem \ref{MainHigherConvex} highlights that it is not only exponentiation that disrupts additive structure, which is the basis for the sum-product problem, but in fact any non-linear function which does not oscillate too wildly on the interval containing $A$. Indeed, if $f$ is sufficiently differentiable and any of its first $k+1$ derivatives has a controlled number of zeroes, then, by the pigeonhole principle, one can pass to a reasonably large subset of $A$ which avoids them all.  

Second, while the dependence on $k$ and $K$ in Theorem \ref{MainHigherConvex} is hardly optimal, the theorem tells us that the condition that $f$ be $k$-convex is necessary and sufficient in order to produce a lower bound of the form $|A|^{k+1}$. Indeed, the $k$'th powers $\{n^k:1\leq n\leq N\}$ can be seen to grow like $N^k$ under iterated sums and differences, but not beyond. On the other hand, incidence geometry techniques for convex functions, initiated in \cite{ENR} and based on the Szemer\'edi-Trotter theorem, seem not to distinguish between convexity and higher convexity. 

Finally, we point out that Theorem \ref{MainHigherConvex} gives new information in the context of the $3$-fold sum-product problem. The resulting inequality
\begin{equation} \label{3fold}
\max \{|A+A-A|, |AA/A| \} \gg \frac{|A|^{3/2}}{(\log |A|)^{3/2}}
\end{equation}
appears to be stronger than the previously known best $3$-fold sum-product  bounds.

\section{Proof of Theorem \ref{BasicHigherConvex}} \label{sec:basic}

The proof is a natural extension of that of Theorem \ref{RSSS} which we will recall. To proceed we introduce some notation. Let $\Delta a_{i}=a_{i+1}-a_i$, so $\Delta$ is the forward difference operator. In this notation, Theorem \ref{RSSS} follows from the observation that
\begin{equation}\label{squeeze0}
    a_i<a_i+\Delta a_j< a_{i+1}
\end{equation}
for $j<i$. The latter condition is guaranteed by insisting that $j<N/2<i$, which is the condition we will adopt in our proof.
\begin{proof}[Proof of Theorem \ref{BasicHigherConvex}]
At the expense of a constant factor, we may assume that $|A|=2^{l}-1$, for some $l$, by truncating $A$ if necessary. 
We proceed by induction on $k$. We will in fact show that the set $2^kA-(2^k-1)A$ contains at least $|A|^{k+1}2^{-k^2}$ elements in the interval $(\min(A), \max(A))$, this stronger statement being better suited to induction. The base case of this induction is Theorem \ref{RSSS}, and follows immediately from inequality (\ref{squeeze0}).

For the inductive step, we let \[(\Delta A)'=\{a_{2}-a_1<\ldots<a_{2^{l-1}}-a_{2^{l-1}-1}\}\] 
and 
\[(\Delta A)''=\{a_{2^{l-1}+1}-a_{2^{l-1}}<\ldots<a_{2^{l}-1}-a_{2^{l}-2}\}.\] 
For $i\geq 2^{l-1}$, the set $a_i+(\Delta A)'$ is a subset of the interval $(a_i,a_{i+1})$ because \[\max((\Delta A)')<\min((\Delta A)'')\leq a_{i+1}-a_i.\] By induction $2^{k-1}(\Delta A)'-(2^{k-1}-1)(\Delta A)'$ contains at least $(2^{l-1}-1)^{k}2^{-(k-1)^2}$ elements in the interval $(a_2-a_1,a_{2^{l-1}}-a_{2^{l-1}-1})$. Thus we can find $(2^{l-1}-1)^k2^{-(k-1)^2}$ distinct elements in the interval $(a_{i},a_{i+1}]$ from the set $a_i+2^{k-1}(\Delta A)'-(2^{k-1}-1)(\Delta A)'$. Since $(\Delta A)'\subseteq A-A$, we have produced a total of \[(2^{l-1}-1)(2^{l-1}-1)^k2^{-(k-1)^2}\geq |A|^{k+1}2^{-k^2}\] distinct elements from the set $2^kA-(2^k-1)A$.
\end{proof}

\section{Proof of Theorem \ref{MainHigherConvex}} \label{sec:main}

To motivate the proof of Theorem \ref{MainHigherConvex}, which is a generalisation of the proof of Theorem \ref{BasicHigherConvex}, it is worth generalising Theorem \ref{RSSS} first. Let us assume for the moment that both $f$ and $f'$ are increasing.

A convex sequence is obtained by applying a convex function $f$ to the interval $[N]=\{1,\ldots,N\}$ to get $a_{i}=f(i)$. In this way, the inequalities 
\[a_i<a_i+a_{j+1}-a_j\leq a_{i+1}\]
can be rewritten as
\begin{equation}\label{squeeze1}
    f(i)<f(i)+f(j+1)-f(j)<f(i+1).
\end{equation}
The significance of adding 1 to the arguments of $f$ is that 1 is a common consecutive difference of $[N]$. The quantity $f(j+1)-f(j)$ is suggestive of a derivative of $f$. Indeed, if $f'$ is increasing then so is \[(\Delta_h f)(x):=f(x+h)-f(x)=\int_0^{h}f'(x+t)dt\]
for any fixed $h>0$.

Now the proof Theorem \ref{RSSS} allows us to deduce that if $A=\{a_1<\ldots<a_N\}$ has many distinct consecutive differences $a_{i+1}-a_i$, then $|A+A-A|$ is large. We can therefore use the bound $|A+A-A|\leq K|A|$, along with a dyadic pigeonhole argument, to deduce that many consecutive differences of $A$ are oft-repeated. Once such a repeated difference, say $h$, is obtained, we can use that $\Delta_h f$ is increasing to generate many inequalities of the form (\ref{squeeze1}). 

Throughout we shall write \[H(A)=\{a_{j+1}-a_j:j=1,\ldots,N-1\}\]
for the set of consecutive differences of $A$, and for $h\in H(A)$, we write \[A_h=\{a_i:a_{i+1}-a_i=h\}.\]
With this notation in hand, if $f'$ is increasing and $a_i,a_j\in A_h$ for some $j<i$ then 
\begin{equation}\label{squeeze2}
    f(a_i)<f(a_i)+f(a_j+h)-f(a_j)\leq f(a_{i}+h),
\end{equation}
generalising (\ref{squeeze1}). If $f$ were decreasing, we would instead have
\begin{equation}\label{squeeze3}
    f(a_i+h)<f(a_i)-f(a_j+h)+f(a_j)\leq f(a_{i}),
\end{equation}
and if $\Delta_h f$ were decreasing in either case, we would insist instead that $j>i$.

We will need the following dyadic pigeonhole lemma, which is a quantitative version of Theorem \ref{RSSS}.
\begin{Lemma}\label{Dyadic}
Let $A=\{a_1<\ldots<a_N\}$ be a set of reals and suppose $N$ is sufficiently large. There is a set $H'\subseteq H(A)$ of size $m$ and an integer $L$ satisfying
\begin{enumerate}[label=(\roman*)]
    \item $Lm\geq N/(3\log_2N)$,
    \item $ |A+A-A| \geq (Lm^2)/2$, and
    \item for each $h\in H'$, $L\leq |A_h|\leq 2L$.
    
\end{enumerate}
\end{Lemma}
\begin{proof}
Let \[i(j)=1+|\{h\in H(A):h<a_{j+1}-a_j\}|.\] In other words, if $H(A)=\{h_1<\ldots<h_k\}$ is ordered, $i(a_{j+1}-a_j)=l$ if $a_{j+1}-a_j=h_l$ (so $i$ is the index). 

Moreover, since \[a_j\leq a_j+h<a_{j+1}\]
for $h\in H\cup\{0\}$ with $h<a_{j+1}-a_j$, it follows that there are $i(j)$ elements of $A+A-A$ in the interval $[a_j,a_{j+1})$, whence
\[|A+A-A|\geq \sum_{j=1}^{N-1}i(j)=\sum_{l=1}^k l|i^{-1}(l)|.\]
For $0\leq t \leq \log N$, let \[I_t=\{l:2^t\leq |i^{-1}(l)|< 2^{t+1}\}\]
Then
\[N-1=\sum_{l}|i^{-1}(l)|\leq 2\sum_{t=0}^{\log_2 N}2^t|I_t|,\]
so for some $t$ we have \[|I_t|2^t\geq \frac{N}{3\log_2 N}.\]
From this,
\[\sum_{l=1}^k l|i^{-1}(l)|\geq \sum_{l\in I_t}l2^t\geq 2^{t-1}|I_t|^2.\]
We let $H'=\{h_l:l\in I_t\}$ and $L=2^t$.
\end{proof}

We now prove Theorem \ref{MainHigherConvex} in the case $k=1$. This means showing that 
\[|2f(A)-f(A)|\gg \frac{|A|^2}{K(\log|A|)^3}.\]
Applying Lemma \ref{Dyadic}, we obtain a subset $H'\subset H(A)$ of size $m$ and an integer $L$. For $h\in H'$ and for $a_i,a_j\in A_h$ we can apply (\ref{squeeze2}), provided $j<i$ and $f'$ is increasing. This gives us $L^2/2$ different sums from $2f(A)-f(A)$ in the intervals $(f(a_i),f(a_i+h))$. These intervals are disjoint for distinct $h$ as $a_i+h=a_{i+1}$ and $f$ is increasing. Adding the contributions from different $h\in H'$, we have produced $mL^2/2$ distinct sums in $2f(A)-f(A)$. By (i) and (ii),
\[|2f(A)-f(A)|\gg \frac{(mL)^2}{m}\gg \frac{|A|^2}{K(\log |A|)^3}.\] In the case that $f$ or $f'$ is decreasing, we can use inequality (\ref{squeeze3}) or insist that $j>i$ as needed.

Before moving on to the general case of Theorem \ref{MainHigherConvex}, we need a bit of notation. Recall that for $h_1\in H(A)$ we write
\[A_{h_1}=\{a_i\in A: a_{i+1} -a_i = h_1\}.\] Inductively, we define
\[A_{h_1,\ldots,h_k}=(A_{h_1,\ldots,h_{k-1}})_{h_k},\]
which means those elements $a\in A_{h_1,\ldots,h_{k-1}}$ whose successor in $A_{h_1,\ldots,h_{k-1}}$ is $a+h_k$. Similarly, we write
\[(\Delta_{h_1}f)(x)=f(x+h_1)-f(x)\]
and inductively define
\[(\Delta_{h_1,\ldots,h_k}f)(x)=(\Delta_{h_1,\ldots,h_{k-1}}f)(x+h_k)-(\Delta_{h_1,\ldots,h_{k-1}}f)(x).\]

\begin{Lemma}\label{Convex}
Let $f$ be a $k$-convex function on $[a,b]$. Then for any $h$ with $0<h<b-a$, $\Delta_{h}f$ is $(k-1)$-convex on $[a,b-h]$. In particular, if $j\leq k$ and $h_1,\ldots,h_j>0$ then $\Delta_{h_1,\ldots,h_j}f$ is strictly monotone. 
\end{Lemma}
\begin{proof}
The first claim holds as \[(\Delta_h f)^{(l)}(x)=f^{(l)}(x+h)-f^{(l)}(x)=\int_0^h f^{(l+1)}(x+t)dt,\]
and if $l\leq k$ then the integrand on the right has fixed sign, and so the sign of the left hand side is fixed as well. It follows that $\Delta_{h_1,\ldots,h_j}f$ is $(k-j)$-convex, and since $k-j\geq 0$, this means its derivative is non-vanishing. 
\end{proof}
%\fxnote{I am not sure what ``integrand constant sign" was supposed to be? Some typo that I don't know how to correct.}

\begin{proof}[Proof of Theorem \ref{MainHigherConvex}]
We will prove the statement by induction on $k$, with the further conclusion that all sums produced lie in the interval $(\min(f(A)),\max(f(A)))$. We will assume that $f$ and $f'$ are both increasing, as the case where either of them is decreasing is handled similarly. The constant $C$ will be chosen so as to close the induction, and may change from line to line. We have already established the case $k=1$. So assume the statement holds for $k-1$ and any set $A$ with corresponding value of $K$.

Apply Lemma \ref{Dyadic} to obtain a set $H'\subset H(A)$ of size $m$ such that for each $h\in H'$ we have a set $A_h$ of size $N_h\geq L$. Now write
\[A_h=\{a_1<\ldots<a_{N_h}\}.\]
The function $\Delta_h f$ is increasing and so the values $\Delta_hf(a_i)$ are as well. Let 
\[A_h'=\{a_i:i\leq N_h/2\},\ A_h''=\{a_i:i\geq N_h/2\}.\]
%\fxnote{There is a typo here and I guess it should be ``...by Lemma 3.2"? But then the intervals do not quite line up?}
Then we have the containment
\[f(a_i)+\Delta_hf(A_h')\subset (f(a_i),f(a_i+h))\]
for each $a_i\in A_h''$. But the left hand side is the image of $A_h'$ under the function $g=f(a_i)+\Delta_h f$. The function $g$ is $(k-1)$-convex on $(\min(A_h'),\max(A_h'))$ by Lemma \ref{Convex}. Now \[|A_h'+A_h'-A_h'|\leq |A+A-A|\leq K|A|\leq (CK^2(\log |A|)^2)|A_h'|,\]
by properties (i) and (ii) of Lemma \ref{Dyadic}. So by induction there are at least \[T=\frac{L^{k}}{(CK^2(\log|A|)^2)^{2^{k}-k-1}(\log|A|)^{2^{k+1}-k-3}}\]
distinct elements from \[2^{k-1}g(A_h)-(2^{k-1}-1)g(A_h)\subseteq 2^kf(A)-(2^{k}-1)f(A)\] in the interval $(f(a_i),f(a_i+h))$.% We can do this separately for each $a_i\in A_h''$, and then repeat the process for each $h\in H'$. %In this way, we generate at least
Now \[L\geq \frac{|A|}{Cm\log |A|}\geq\frac{|A|}{CK(\log |A|)^2}\]
so that 
\[T\geq \frac{|A|^k}{C^{2^k-1}K^{2^{k+1}-k-2}(\log|A|)^{2^{k+2}-k-5}}.\]
This induction now applies to each $a_i\in A_h''$, independently for each $h\in H'$, and the elements of $2^{k}f(A)-2^{k-1}f(A)$ constructed belong to disjoint sub-intervals of $(\min(A),\max(A)]$. The total number of elements constructed is
\[LmT\gg \frac{|A|}{C\log|A|}\cdot \frac{|A|^k}{C^{2^k-1}K^{2^{k+1}-k-2}(\log|A|)^{2^{k+2}-k-5}}=\frac{|A|^{k+1}}{C^{2^k}K^{2^{k+1}-k-2}(\log|A|)^{2^{k+2}-k-4}}\]
and the induction closes for appropriately large $C$.
\end{proof}

\section{Applications} \label{sec:applications}
Of particular relevance in this section is the following general form of the Erd\H{o}s-Szemer\'{e}di sum-product conjecture \cite{ES}:  for all $\eps >0$ and every integer $k \geq 2$ there exists a constant $c=c(k, \eps)>0$ such that for all $A \subset \mathbb N$
\[
\max \{ |kA|,|A^{(k)}| \} \geq c|A|^{k-\eps}.
\] 
In the above inequality, $A^{(k)}$ denotes the $k$-fold product set
\[
A^{(k)}=\{x_1\cdots x_k : x_1,\dots, x_k \in A \}.
\]
The conjecture is widely believed to hold in the more general setting of $A \subset \mathbb R$. The question is wide open even in the most studied the case of $k=2$. Little is known for larger values of $k$, with the exception of the milestone FPMS-paper by Bourgain and Chang \cite{BC} for sets of natural numbers, whose results were recently reproved in a beautiful way and somewhat strengthened by  P\'alv\"olgyi  and Zhelezov \cite{PZ}.

In order to make significant progress with these difficult problems, it is necessary at least to understand what happens in the extreme case when one of the sum set or product set is small. But even with this strong additional information, not everything is clear, especially over $\mathbb R$. There has been progress in proving that few products implies many iterated sums, and it is known that $|kA|$ grows arbitrarily large when $|AA| \leq K|A|$ and $K$ is small (in particular, effective results are known for $K=|A|^{\eps}$ for $\eps>0$ sufficiently small). The best result (over $\mathbb R$) in this direction is due to Bush and Croot \cite{BuCr}.

\subsection{Few sums implies many iterated products}

Our first application concerns product sets of $A$ when the sum set is small. For $k=2$ an optimal result
\[
|A+A| \leq K|A| \Rightarrow |AA| \geq |A|^{2-\eps}K^{-C}
\]
follows from a simple application of the Szemer\'{e}di-Trotter Theorem. See \cite{ER}, where $C=4$; it is possible to lower it down to $8/3$. It is not obvious how to iterate this to get bigger growth for more products, and even the result
\[
|A+A| \leq K|A| \Rightarrow |AAA| \geq |A|^{2+\eps}K^{-C}
\]
was proved only recently \cite{RNS}.

Theorem \ref{MainHigherConvex} implies the following result that \textit{few sums implies unbounded iterated products}, part (1) of Corollary \ref{Applications}.

\begin{proof}[Proof of Corollary \ref{Applications}, (1)]
Let $f(x)=\log (x)$ and note that $f$ and all of its derivatives are non-zero on $(0,\infty)$. Apply Theorem \ref{MainHigherConvex} with this function $f$, which is $k$-convex for any $k$.

By the Pl\"{u}nnecke inequality
\[
|A+A-A| \leq |A|^{1+3\delta}.
\] 
Theorem \ref{MainHigherConvex} therefore gives
\[
|2^kf(A)-(2^{k}-1)f(A)|\gg \frac{|A|^{k+1}}{(C|A|^{3\delta})^{2^{k+1}-k-2}(\log |A|)^{2^{k+2}-k-4}}.\]
By taking $\delta$ to be sufficiently small with respect to $\eps$ and $k$, the denominator above can be made smaller than $|A|^\eps$, provided $A$ is sufficiently large.
Since $f(A)=\log(A)$, sums of $f(A)$ correspond with products of $A$, and the statement follows.
\end{proof}

\subsection{Products and shifts}

An important variant of the sum-product phenomenon is the idea that additive perturbations destroy multiplicative structure. For instance, consider the following result of Shkredov \cite{Sh}: for all finite $A \subset \mathbb R$,
\begin{equation} \label{shifts}
|AA|^4|(A+1)(A+1)| \gg \frac{|A|^6}{\log |A|}.
\end{equation}
In particular, if $|AA| \leq K|A|$ it follows that the shifted product set $(A+1)(A+1)$ is close to maximal in size. Note that the value of the shift in \eqref{shifts} can be replaced by an arbitrary non-zero $\lambda$ by applying the inequality to $\lambda^{-1}\cdot A$.

Similar statements to \eqref{shifts} can be deduced from results on convexity and sum sets. For instance, let $X \subset \mathbb R$ and apply \eqref{ENRgen} with
\[A=\log(X)=B,\, C=\log(X+1),\, f(x)=\log(e^x+1)
\]
it follows that
\[
|AA|^2|(A+1)(A+1)|^2 \gg |A|^5.
\]

Theorem \ref{MainHigherConvex} gives similar statement about iterated growth, which is (2) of Corollary \ref{Applications}.
\begin{proof}[Proof of Corollary \ref{Applications}, (2)]
Let $B=\log A$, then using Pl\"unnecke's inequality, we have \[|B+B-B|\leq |B|^{1+3\delta}.\]
Let $f(x)=\log(1+e^x)$. Calculating derivatives of $f$ shows that $f$ is $k$-convex for any $k$. It therefore follows that
\[
|(A+1)^{(2^k)}/(A+1)^{(2^k-1)} | =|2^kf(B)-(2^{k}-1)f(B)|\geq \frac{|A|^{k+1}}{(C|A|^{3\delta})^{2^{k+1}-k-2}(\log |A|)^{2^{k+2}-k-4}}.
\]
By taking $\delta$ to be sufficiently small with respect to $\eps$ and $k$, the denominator is at most $|A|^{\epsilon}$, as required.
\end{proof}

We anticipate future applications of the ideas in this paper to give new results about growth of products under shifts and further consequences. We present one application here, which is an analogue of the main result from \cite{RNS}, and gives a positive answer to a question raised in \cite{RNW}.

\begin{Theorem}
Let $A$ be a set of positive reals and suppose that $|AA| \leq K|A|$. Then there are positive constants $c, C$ such that
\[
|(A+1)(A+1)(A+1)| \gg \frac{|A|^{2+c}}{K^C}.
\]
\end{Theorem}

\begin{proof}
This argument closely follows the proof of Theorem 3.2 in \cite{RNS}, and a slightly more detailed account can be found there. The important new information comes from an application of Theorem \ref{MainHigherConvex}.

\begin{Claim}
Let $A \subset \mathbb R$ and $|AA| \leq K|A|$. Then
\begin{equation} \label{claimed}
\left |\frac{(A+1)(A+1)(A+1)(A+1)}{(A+1)(A+1)(A+1)} \right | \gg \frac{|A|^3}{K^{18}(\log |A|)^{14}}
\end{equation}

\end{Claim}

\begin{proof} Apply Theorem \ref{MainHigherConvex} with $k=3$ and
\[B=\log A,\ f(x)=\log(e^{x}+1).
\]
Then $|B+B-B|=|AA/A|$ and $|4f(B)-3f(B)|=|(A+1)^{(4)}/(A+1)^{(3)}|$. The claim follows.

\end{proof}
Further applications of Pl\"{u}nnecke-Ruzsa give
\begin{align*}
\frac{|A|^3}{K^{18}(\log |A|)^{14}} \ll \left| \frac{(A+1)(A+1)(A+1)(A+1)}{(A+1)(A+1)(A+1)(A+1)}\right | &\leq \frac{|(A+1)(A+1)(A+1)(A+1)|^4}{|(A+1)(A+1)|^3} 
\\&\leq \frac{|(A+1)(A+1)(A+1)|^{16}}{|(A+1)(A+1)|^{15}}
\end{align*}
It follows from an application of \eqref{shifts} that
\[
|(A+1)(A+1)(A+1)|^{16} \gg \frac{|A|^{33}}{K^{78}(\log|A|)^{29}},
\]
which completes the proof.
\end{proof}

\subsection{$|A+A|$ versus $|A^k+A^k|$} 
Raising the elements of $A$ to a power ought to disrupt any additive structure present in the set $A$. As in the preceding sections, we can deduce Corollary \ref{Applications} part (3) from Theorem \ref{MainHigherConvex}.
\begin{proof}[Proof of Corollary \ref{Applications}, (3)]
Using the Pl\"unnecke inequality, we have $|A+A-A|\leq |A|^{1+3\eps}$. We deduce the conclusion by applying Theorem \ref{MainHigherConvex} to $A$ with the function $f(x)=x^k$. The same is true if $x^{k}$ is replaced by any power of $x$ other than $0,\ldots,k-1$. 

If $f(x)$ is instead a polynomial of degree $k$ then $f$ and its derivatives combine for at most $k!$ roots. We can select a subset $A'$ of size $|A|/k!$ which lies in an interval not containing any of these roots, and adjust $\eps$ if necessary to overcome the factor $k!$.
\end{proof}
The number of sums and differences of $k$'th powers here is an important problem. The case where $K$ is minimized, i.e. where $A=[N]$, reduces to Waring's problem. In this setting, quantitative questions become far more delicate. This is highlighted by the fact that even in the simplest case $k=3$ and $A=[N]$, the true order of magnitude of $|A+A+A|$ is unknown. The best result is $|A+A+A|\gg N^{3\beta}$ with $\beta=0.91709477$, due to Wooley \cite{W}.

\section*{Acknowledgements}Oliver Roche-Newton was partially supported by the Austrian Science Fund FWF Project P 30405-N32. Misha Rudnev is partially supported by the Leverhulme Trust Grant RPG-2017-371. We are grateful to Antal Balog, Peter Bradshaw, Brendan Murphy and Audie Warren for helpful discussions.

\end{document}